\documentclass[oneside,12pt]{article}

\usepackage{geometry}
\usepackage{tikz}
\usepackage{float}
\usepackage{multirow}
\usepackage{mathtools}
\usepackage[font=small,labelfont=bf]{caption} 
\usepackage{graphicx}
\usepackage{amssymb}
\usepackage{amsmath}
\usepackage{bm}

\usepackage{enumitem}
\setitemize{label=\scriptsize{$\blacksquare$}}
\setenumerate{label=(\alph*), leftmargin=2em}

\usepackage{amsthm}

\newtheorem{theorem}{Theorem}
\newtheorem{lemma}[theorem]{Lemma}
\newtheorem{proposition}[theorem]{Proposition}

\newtheorem{remark}[theorem]{Remark}
\newtheorem{example}[theorem]{Example}

\theoremstyle{definition}

\newtheorem{definition}[theorem]{Definition}

\newcommand{\X}{\mathbb{X}}
\newcommand{\Y}{\mathbb{Y}}

\newcommand{\signal}{x}
\newcommand{\data}{y}

\newcommand{\R}{\mathbb{R}}
\newcommand{\N}{\mathbb{N}}
\newcommand{\M}{\mathbf{M}}

\newcommand\snorm[1]{\lVert#1\rVert}
\newcommand\set[1]{{\{#1\}}}
\newcommand\norm[1]{{\Vert#1\Vert}}
\newcommand\inner[2]{\langle#1,#2\rangle}
\newcommand\abs[1]{\vert#1\vert}
\newcommand\hnorm[1]{\biggl\Vert#1\biggr\Vert}
\newcommand{\scp}[2]{\langle #1,#2\rangle}

\newcommand{\fourier}{\mathcal{F}}
\newcommand{\TT}{\mathcal{T}}

\newcommand{\skl}[1]{(#1)}

\newcommand{\Po}{\mathbf{P}}
\newcommand{\Ko}{\mathbf{K}}
\newcommand{\Bo}{\mathbf{B}}
\newcommand{\Freg}{\mathbf{F}}
\newcommand{\Ro}{\mathbf{R}}

\newcommand{\eps}{\epsilon}
\newcommand{\al}{\alpha}
\newcommand{\supp}{\operatorname{supp}}
\newcommand{\ran}{\operatorname{ran}}
\newcommand{\dom}{\operatorname{dom}}
\def\plus{{\boldsymbol{\ddag}}}
\newcommand{\la}{\lambda}
\newcommand{\La}{\Lambda}
\newcommand{\ka}{\kappa}

\newcommand{\ao}{{\boldsymbol{a}}}
\newcommand{\bom}{{\boldsymbol{\omega}}}
\newcommand{\kao}{{\boldsymbol{\kappa}}}
\newcommand{\vo}{{\boldsymbol{v}}}
\newcommand{\uo}{{\boldsymbol{u}}}
\newcommand{\buo}{{\boldsymbol{\bar{u}}}}
\newcommand{\U}{\mathbb{U}}
\newcommand{\edot}{\,\cdot\,}
\newcommand{\sph}{\mathbb{S}}

\usepackage{soul}
\usepackage{xcolor}
\colorlet{lred}{red!40}
\colorlet{lgreen}{green!40}
\colorlet{lblue}{blue!40}
\definecolor{bananamania}{rgb}{0.98, 0.91, 0.71}

\allowdisplaybreaks

\newcommand{\dt}{\operatorname{d}\!t}
\newcommand{\dxi}{\operatorname{d}\!\xi}
\newcommand{\dsigma}{\operatorname{d}\!\sigma}
\newcommand{\dtheta}{\operatorname{d}\!\theta}
\newcommand{\dx}{\operatorname{d}\!x}
\newcommand{\ds}{\operatorname{d}\!s}

\newcommand{\riesz}{\boldsymbol{\Omega}}
\newcommand{\Z}{\mathbb{Z}}

\usepackage{hyperref}

\numberwithin{equation}{section}
\numberwithin{figure}{section}
\numberwithin{theorem}{section}

\title{Regularization of Inverse Problems by Filtered Diagonal Frame Decomposition}

\author{Andrea Ebner\thanks{Department of Mathematics, University of Innsbruck,
Technikerstrasse 13, 6020 Innsbruck, Austria, \{andrea.ebner, johannes.schwab, markus.haltmeier\}@uibk.ac.at} \and J\"urgen Frikel \thanks{Department of Computer Science and Mathematics, Galgenbergstra{\ss}e 32, D-93053 Regensburg, Germany, juergen.frikel@oth-regensburg.de} \and Dirk Lorenz \thanks{Institute of Analysis and Algebra, Technical University of Braunschweig, Universit\"atsplatz 2, 38106 Braunschweig, Germany, d.lorenz@tu-braunschweig.de} \and Johannes Schwab \footnotemark[1] \and Markus Haltmeier\footnotemark[1]  }

\begin{document}

\maketitle

\begin{abstract}
Inverse problems are at the heart of many practical problems such as medical image reconstruction or non-destructive evaluation. A characteristic feature of inverse problems is their instability with respect to data perturbations. In order to stabilize the inversion process, regularization methods have to be developed and applied. In this paper, we introduce and analyze the concept of filtered diagonal frame decomposition, which extends the classical filtered singular value decomposition (or spectral filtering) to the case of frames. The use of frames as generalized singular systems allows for a better adaption to a given class of potential solutions of the inverse problem. This is also beneficial for problems where the SVD is not available analytically. We show that filtered diagonal frame decompositions provide convergent regularization methods. Moreover, we derive convergence rates under source conditions and prove order optimality when the frame under consideration is a Riesz basis. Our analysis applies to unbounded and bounded forward operators. As a practical application of our tools we study filtered diagonal frame decompositions for inverting the Radon transform as an unbounded operator on $L^2(\R^2)$.

\bigskip\noindent\textbf{keywords}
Inverse problems, frame decomposition, Moore-Penrose inverse, 
convergence analysis, convergence rates, Radon transform, computed tomography
\end{abstract}

\section{Introduction}
\label{sec:intro}

This paper is concerned with solving inverse problems of the form
\begin{equation}\label{eq:ip}
	\data  =  \Ko\signal+ z \,,
\end{equation}
where $\Ko\colon \dom(\Ko) \subseteq \X\to \Y$ is a closed densely defined linear operator between Hilbert spaces $\X$ and  $\Y$, and $z$ denotes the data distortion that satisfies $\norm{z} \leq \delta$ for some noise level $\delta\geq 0$. A characteristic property of inverse problems is that they are ill-posed \cite{engl1996regularization,scherzer2009variational}. This means that the solution of \eqref{eq:ip} is either not unique or is unstable with respect to  perturbations of the right-hand side.  Note that our treatment includes the case of unbounded forward operators. On the one hand this does not make proofs significantly more complicated than in the case of bounded forward operators, and on the other hand unbounded forward operators are important for practically relevant inverse problems. For example, the Radon transform is well known to be unbounded as an operator on $L^2(\R^2)$ which is the natural Hilbert space where wavelet frames are defined. Restricting  to functions vanishing outside a bounded domain would make the Radon transform bounded but would also require to adjust the underlying wavelets to the boundary. Further, on bounded domains,  main theoretical tools such as the Fourier slice identity are not directly applicable.

Arguably, the  theory of solving inverse problems of the form \eqref{eq:ip}
 is quite well developed.  Especially, the class of filter based methods 
gives a wide range of solution schemes. Assuming that $\Ko$ has a singular value decomposition (SVD) 
$\Ko
	 =
	\sum_{n \in \N}  \sigma_n   \inner{\edot}{u_n}
	v_n$,
	these methods take one of the following equivalent forms
 \begin{align} \label{eq:filtersvd1}
 	\Freg_\al \data
	& =
	\sum_{n \in \N} g_\al( \sigma_n^2 )  \inner{\Ko^* \data}{u_n}
	u_n  \\ \label{eq:filtersvd2}
	\Freg_\al \data 
	& =
	\sum_{n \in \N} f_\alpha( \sigma_n )
	\inner{ \data}{v_n}  u_n  \,.
\end{align}
Here  $(g_\al)_{\al >0}$  is a family  of bounded functions converging pointwise to $1/\lambda $ as $\alpha  \to 0$ and $f_\alpha(\sigma)
\coloneqq  \sigma g_\alpha( \sigma^2)$.    Note that the form \eqref{eq:filtersvd1} derives from functional calculus  applied to $ g_\al (\Ko^* \Ko)   \Ko^*$ whereas \eqref{eq:filtersvd2} can be naturally generalized to frame decompositions instead of an SVD. The form \eqref{eq:filtersvd2} can be seen as regularized version of the  SVD based   formula $\Ko^\plus  \data
	 =
	\sum_{n \in \N}
	\sigma_n^{-1}
	\inner{ \data }{v_n}  u_n $ for the  Moore-Penrose pseudo inverse $\Ko^\plus$ of $\Ko$. The analysis of such  regularization methods can be found, for example, in \cite{engl1996regularization,groetsch1984theory} in the case  of bounded $\Ko$; compare \cite{hofmann2009regularization} for the case of unbounded forward operators. For general background on pseudo inverses, see, for example,  \cite{ben2003generalized}. 

The SVD cannot be adapted to the underlying signal class and therefore  is not always a good representation for various kinds of inverse problems. Instead, certain diagonal frame decompositions generalizing the SVD are  better suited because  the defining  frames  can be adjusted to a particular application \cite{candes2002recovering,donoho95nonlinear,frikel2019sparse}. To the best of our knowledge, filter  based methods based on diagonal frame decompositions have not been rigorously studied in the context of regularization theory. (Note that after initial submission of our manuscript a related analysis appeared in \cite{hubmer2022regularization}. Most notably, opposed  to that paper, our analysis allows unbounded forward operators and considers order optimality and characterization  of ill-posedness via the frame decompositon. On the other hand, \cite{hubmer2022regularization}  additionally considers the discrepancy principle which we do not address.) This paper addresses  this issue and develops a regularization theory for diagonalizing systems including the SVD based filter methods as special case.

\subsection{Filtered diagonal frame decomposition}

A diagonal frame decomposition (DFD)  for the operator $\Ko$ consists of  a frame  $(u_\la)_{\la \in \La}$ of  $(\ker \Ko)^\perp$, a frame $(v_\la)_{\la \in \La}$ of $\overline{\ran \Ko}$ and a sequence of positive numbers $(\ka_\la)_{\la \in\La}$ such that the pseudo inverse  of $\Ko$ has the form (see Section~\ref{ssec:dfd})  
\begin{equation}\label{eq:pseudo} 
\forall \data \in \dom(\Ko^\plus) = \ran(\Ko) \oplus \ran(\Ko)^\bot \colon
\quad 
  \Ko^\plus \data
	 =
	\sum_{\la \in \La}
	\frac{1}{\ka_\la}
	\inner{ \data }{v_\la} \bar u_\la \,.
\end{equation}
Here $(\bar u_\la)_\la$ is any dual frame of  $(u_\la)_{\la \in \La}$ and  $\ka_\la > 0$ are the generalized singular values. Equation \eqref{eq:pseudo}  is a generalization of the SVD allowing frames as  non-orthogonal  generalized singular systems $(u_\la)_\la$ and
$(v_\la)_\la$. Moreover,  both systems are in general overcomplete, which is another main  reason for using frames. Opposed to the SVD, many different DFDs for a given operator can  exist and the quasi-singular systems can be adapted to a  particular signal class.

In the case of ill-posed problems where $\Ko^\plus$ is unbounded, regularization techniques have to be applied in order to approximately but stably solve \eqref{eq:ip}.  Based on a DFD of the forward operator, in this paper, we consider filtered DFDs defined as
\begin{equation*} 
 \Freg_\al \data  
  \coloneqq \sum_{\la \in \La}
 f_\al(\ka_\la) \inner{\data}{v_\la} \bar u_\la \,.
\end{equation*}
Here  $(f_\al)_{\al >0}$  is a family of functions converging pointwise to $1/\ka $ as $\alpha  \to 0$ (more precisely, a regularizing filter; see Definition~\ref{def:refF}).
In case  we take the SVD as the DFD then the filtered DFD reduces to classical filter based regularization.  However,  the filtered DFD contains other interesting special cases. In particular, taking $(u_\la)_{\la \in \La}$ as  wavelet, curvelet and shearlet system yields DFDs for  image reconstruction   \cite{candes2002recovering,colonna2010radon,donoho95nonlinear,frikel2019sparse}.
We also point out that such systems are often used in variational regularization schemes \cite{daubechies2004iterative,dicken1996wavelet,elad2017analysis,grasmair2008sparse,grasmair2011necessary,lorenz2008convergence,rieder1997wavelet} which are related but different from the approach followed in this paper. In the context of variational regularization, regularized solutions are constructed as minimizers of a generalized Tikhonov functional formed by adding a frame-dependent regularizer to the operator-dependent data fitting term.

\subsection{Outline}

In Section \ref{sec:dfd} we introduce and study the concept of a DFD and  relate   the ill-posedness of the inverse problem \eqref{eq:ip} to the decay of the quasi-singular values. In  Section \ref{sec:reg} we introduce filtered DFDs to account for the ill-posedness of \eqref{eq:ip}. We show that filtered DFDs yield regularization methods and we derive convergence rates  under  source-type conditions on the unknowns to be recovered. In Section \ref{sec:radon} we present and implement filtered DFDs for  stable  Radon transform  inversion  as  practically relevant example from  medical image reconstruction. The paper concludes with  a short discussion and outlook given  in Section \ref{sec:end}.

\section{Operator inversion by diagonal frame decomposition}
\label{sec:dfd}

Throughout this paper $\X$ and $\Y$ denote Hilbert spaces over $\mathbb{K} \in \{\R, \mathbb{C} \}$  and $\Ko \colon  \dom(\Ko) \subseteq  \X \to \Y$ a closed, densely defined linear operator. Note that we do not assume the operator $\Ko$ to be  bounded. For example, this allows to include the Radon transform on $L^2(\R^2)$ in our setting; see Section~\ref{sec:radon}.
In this section, we introduce diagonal frame decompositions (DFDs) which
in the following sections will be used to regularize the inverse problem defined by the forward operator $\Ko$.

\subsection{Frames}

We start by briefly recalling some basic facts about frame theory \cite{Ch02,mallat2009wavelet,adcock2019frames}.
A family $\uo = (u_\la)_{\la \in \La} \in \U^\La$ where $\La$ is an at most countable index set is called frame for the  Hilbert space $\U$ if there are constants $A,  B>0$ such that
\begin{equation}
\forall x \in \U \colon \quad A  \norm{x}^2 \leq \sum_{\la \in \La} \abs{\inner{u_\la}{x}}^2 \leq B \norm{x}^2 \,.
\end{equation}
The constants $A$ and $B$ are called lower  and upper frame bounds of $\uo$, respectively. The frame is called tight if $A = B$ and exact if  it fails to be a frame whenever any single element is deleted from the sequence $(u_\la)_{\la \in \La}$. A frame that is not a Riesz basis is said to be overcomplete.

\begin{definition}[Analysis and synthesis operator]
Let $\uo = (u_\la)_{\la \in \La}$  be a frame for the Hilbert space $\U$.
The analysis and synthesis operator of $\uo$, respectively, are  defined by
\begin{align}
&\TT_\uo \colon \U \to \ell^2(\La): x \mapsto (\inner{x}{u_\la})_{\la \in \La}\\
&\TT^*_\uo \colon \ell^2(\La) \to \U \colon (c_\la)_{\la \in \La} \mapsto \sum_{\la \in \La} c_\la u_\la \,.
\end{align}
\end{definition}

One easily  verifies  that $\TT_\uo$ and $\TT_\uo^*$ are linear bounded operators and the   synthesis operator  $\TT_\uo^*$ is the adjoint of the analysis operator $\TT_\uo$.

\begin{definition}[Dual frame]
Let $\uo = (u_\la)_{\la \in \La}$ be a frame for the Hilbert space $\U$. A frame $\buo = (\bar u_\la)_{\la \in \La}$ for $\U$ is called a dual frame of $\uo$ if the following duality condition holds:
\begin{equation}\label{eq:dualF}
\forall x \in \U \colon \quad x=\sum_{\la \in \La} \inner{x}{u_\la} \bar u_\la = \TT_{\buo}^* \TT_\uo x \,.
\end{equation}
\end{definition}
Every frame has at least one dual frame and if the frame $\uo$ is over-complete, then there exist infinitely many dual frames of $\uo$.

\begin{definition}[Norm bounded frames]
Let $\U$ be a Hilbert space and $\uo$ a frame for $\U$. We call $\uo$ norm bounded from below if there exists a constant $a>0$ such that $\inf_{\la \in \La} \norm{u_\la} \geq a$.
\end{definition}

Note that every frame is already norm  bounded from above. In fact, the upper frame condition implies $\norm{u_\la}^4 = \abs{\inner{u_\la}{u_\la}}^2 \leq \sum_{\mu \in \La} \abs{\inner{u_\la}{u_\mu}}^2 \leq B \norm{u_\la}^2 $ which gives  $\sup_{\la \in \La} \norm{u_\la} \leq \sqrt{B}$.
On the other hand one easily constructs  examples of frames that are not norm bounded from below.

\subsection{Diagonal frame decomposition}
\label{ssec:dfd}

We use the  following notion extending the wavelet-vaguelette decomposition (WVD) and biorthogonal  curvelet decomposition to more general frames. It will allow  us to unify  and extend existing filter based regularization methods to the frame case.

\begin{definition}[Diagonal frame decomposition, DFD] \label{def:dfd}
Let $\Ko \colon \dom(\Ko) \subseteq \X \to \Y$ be a closed and densely defined linear operator, and $\La$ an at most countable index set. We call $(\uo, \vo, \kao) = (u_\la, v_\la, \ka_\la)_{\la \in \La}$ a diagonal frame decomposition (DFD) for the operator $\Ko$ if the following holds:

\begin{enumerate}[itemindent =2em, leftmargin =2em,  label=(D\arabic*)]
\item\label{D1}  $(u_\la)_{\la \in \La}$ is a frame for $(\ker{\Ko})^{\perp} \subseteq \X$,
\item\label{D2}  $(v_\la)_{\la \in \La}$ is a frame for $\overline{\ran\Ko}\subseteq \Y$,
\item\label{D3}  $(\kappa_\la)_{\la \in \La}\in (0, \infty)^\La$ satisfies the quasi-singular relations
\begin{equation} \label{eq:qsr}
\forall \la \in \La \colon \quad \Ko^* v_\la = \ka_\la u_\la \,.
\end{equation}
\end{enumerate}
We call $(\ka_\la)_{\la \in \La}$ the quasi-singular values and $(u_\la)_{\la \in \La}$, $(v_\la)_{\la \in \La}$  the   corresponding quasi-singular systems. 
\end{definition}

In the case $\uo$  is an orthonormal wavelet  basis, then   the DFD reduces to the WVD  introduced in  \cite{donoho95nonlinear}.  A WVD decomposition has been constructed for the classical computed tomography modeled by  the two-dimensional Radon transform see \cite{donoho95nonlinear}. In the case of the two-dimensional Radon transform, a  biorthogonal curvelet decomposition was constructed in \cite{candes2002recovering}. In \cite{colonna2010radon}, the authors derived biorthogonal shearlet decompositions for two- and three-dimensional Radon transforms. The  limited data case has been studied in \cite{frikel2013sparse}. 

 Note that the quasi-singular relations in \eqref{eq:qsr}  imply that $v_\la  \in \dom(\Ko^*)$ and $u_\la  \in \ran(\Ko^*)$ which in the unbounded case are abstract smoothness requirements. Interestingly,  opposed to the SVD case, a DFD does not  require $v_\la \in \ran (\Ko)$ in general.

\begin{remark}[DFDs in the ONB case]
Consider the special case where  $\uo$ is an orthonormal basis (ONB) and let $(\bar v_\la)_{\la}$ be a dual frame of  $\vo$. The quasi-singular relations in this case imply $\Ko u_\la = \ka_\la \bar v_\la$ and thus $\bar v_\la \in \ran(\Ko)$ and $u_\lambda \in \dom(\Ko)$ for all $\la \in \La$.  Further, one can also check that the frames $\vo$ and $\bar \vo$ are biorthogonal, $\inner{v_\la}\bar v_\mu= \delta_{\la \mu}$.This  in turn implies that $\vo$ is a Riesz basis (see \cite{Ch02}) and that $\bar\vo$ is the dual Riesz basis uniquely determined by $\vo$. 
\end{remark}

\begin{remark}[Multiplication operators on $\ell^2$]
For  any sequence $ \ao = (a_\la)_{\la \in \La} \in \R^\La$ define the pointwise multiplication operator
\begin{equation*}
\M_\ao \colon   \dom(\M_\ao) \subseteq \ell^2(\La)  \to \ell^2(\La) \colon (c_\la)_{\la \in \La}  \mapsto (a_\la c_\la)_{\la \in \La} 
\end{equation*}
with domain $\dom(\M_\ao) \coloneqq \{(c_\la)_{\la \in \La} \in \ell^2(\La)\mid (a_\la c_\la)_{\la \in \La} \in \ell^2(\La) \}$. Then  $\M_\ao$ is  closed and densely defined, and bounded if and only  if $ \ao$ is bounded.
\end{remark}

\begin{remark}[DFD as frame-based factorization]
Let $(\uo, \vo, \kao)$ be a DFD for  $\Ko$. Then \eqref{eq:qsr} is  equivalent to   $ \inner{\Ko^* v_\la}{x}  = \ka_\la \inner{u_\la}{x}$ for all $ \lambda \in \La$  and all $x \in \dom(\Ko)$. Moreover,   $\ran (\TT_\uo|_{\dom(\Ko)})  =  \{ (\inner{u_\la}{x})_{\la \in \La} \mid x \in \dom(\Ko) \} \subseteq \dom(\M_\kao)$.  Hence   \eqref{eq:qsr} is equivalent to  $\TT_\vo \Ko = \M_\ka \TT_\uo |_{\dom(\Ko)}$.
\end{remark}

\begin{remark}[Moore-Penrose inverse] 
Recall that  $\Ko$ is closed and densely defined but potentially unbounded. For such operators, the Moore-Penrose inverse  $\Ko^\plus \colon \dom(\Ko^\plus)  \subseteq \Y  \to \X $ with $ \dom(\Ko^\plus) \coloneqq \ran(\Ko)  \oplus  \ran(\Ko)^\bot$ is defined as in the case of bounded forward operators, and is  closed with dense domain  \cite[Theorem~2.12]{groetsch2006stable}. For $\data \in \dom(\Ko^\plus)$,   $\Ko^\plus \data$ is uniquely characterized either as the unique solution of $\Ko x =  \Po_{\overline{\ran(\Ko)}} \data$  in  $\dom(\Ko) \cap \ker(\Ko)^\bot$ or the unique least-squares solution of $\Ko x =  \data$ having  minimal norm.   
\end{remark}

\begin{theorem}[Moore-Penrose inverse via DFD]\label{thm:DFDinvPinv}
Let $(\uo,\vo,\kao)$ be a DFD for $\Ko$ and $\buo = (\bar u_\la)_{\la \in \La}$ be a dual frame of $\uo$. Then 
\begin{equation} \label{eq:DFDinvPinv}
\forall  y \in \dom(\Ko^\plus) \colon
\quad 	
         \Ko^\plus  \data
    = 	\sum_{\la \in \La}
	\frac{1}{\ka_\la}
	\inner{ \data }{v_\la} \bar u_\la \,.
\end{equation}
Equivalentely, $\Ko^\plus = \TT_{\buo}^* \M_{1/\kao} \TT_\vo |_{\dom(\Ko^\plus)} $ where $1/\kao$ denotes the pointwise inverse of $\kao$.  
\end{theorem}

\begin{proof}
For any $\data \in \dom(\Ko^\plus) = \ran(\Ko) \oplus \ran(\Ko)^\perp$ define
$\Bo \data \coloneqq \sum_{\la \in \La}
 \ka_\la^{-1} \inner{ \data }{v_\la} \bar u_\la$. We will show that the mapping
$\Bo \colon \dom(\Ko^\plus) \subseteq \Y \to \X \colon \data \mapsto \Bo \data$ equals the Moore-Penrose inverse. For that purpose note that any element in $\dom (\Ko^\plus)$ has the unique representation  $\data = \Ko \signal^\plus + \data^\perp$ where
$\signal^\plus \in \ker(\Ko)^\perp \cap \dom(\Ko)$  and $\data^\perp \in \ran(\Ko)^\perp$.
The identity  $\ka_\la^{-1}
	\inner{ \data }{v_\la}  = \inner{ \signal^\plus }{u_\la} $ shows that $\Bo \data$ is well defined as absolutely convergent sum. Further,
\begin{multline}
\Bo y =
\sum_{\la \in \La}
	\frac{1}{\ka_\la}
	\inner{ \data }{v_\la} \bar u_\la
	=
	\sum_{\la \in \La}
	\frac{1}{\ka_\la}
	\inner{\Ko \signal^\plus }{v_\la} \bar u_\la	
	=
	\sum_{\la \in \La}
	\frac{1}{\ka_\la}
	\inner{  \signal^\plus }{\Ko^* v_\la} \bar u_\la	\\
	=
	\sum_{\la \in \La}
	\inner{\signal^\plus }{u_\la} \bar u_\la
	= \signal^\plus = \Ko^\plus  y \,.	
\end{multline}
Here we used the definition of $\Bo$, the fact  that $v_\la \in \overline{\ran(\Ko)}$, the   quasi-singular relation \eqref{eq:qsr}, and the fact that $\buo$ is a dual frame of $\uo$
for $(\ker{\Ko})^{\perp}$.
\end{proof}

\subsection{Ill-posedness and quasi  singular values}

Typical inverse problems are unstable in the sense that the Moore-Penrose inverse is unbounded. It is well known that the Moore-Penrose inverse of an operator having a SVD is  bounded if and only if  the singular values do not accumulate  at zero.  Below  we show that a similar characterization holds  for the quasi-singular values in a DFD.

\begin{theorem}[Characterization of ill-posedness via DFD]\label{thm:char}
Let $(\uo, \vo, \kao)$ be a DFD of $\Ko$. Then  the following assertions hold.
\begin{enumerate}
\item\label{thm:char1} $\inf_{\la \in \La}  \ka_\la >0$ $\Rightarrow$  $\Ko^\plus$ is bounded.
\item\label{thm:char2} $\vo$ norm bounded from below $\wedge$ $
\Ko^\plus \text{ bounded } \Rightarrow \, \inf_{\la \in \La}  \ka_\la >0$.\end{enumerate}
\end{theorem}

\begin{proof}
 \ref{thm:char1}
  Let $\buo$ be a dual frame of $\uo$. Then, for  every $\data \in \dom (\Ko^\plus)$ we have
\begin{multline*}
\norm{\Ko^\plus y}^2= \norm{\sum_{\la \in \La} \ka_\la^{-1} \inner{y}{ v_\la } \bar u_\la}^2 \leq
\snorm{\TT_{\buo}^*}^2 \sum_{\la \in \La} \abs{\ka_\la^{-1} \inner{y}{ v_\la }}^2
\\
\leq \frac{\snorm{\TT_{\buo}^*}^2 }{(\inf_{\la \in \La}  \ka_\la)^2} \sum_{\la \in \La} \abs{\inner{y}{ v_\la }}^2
\leq \frac{\snorm{\TT_{\buo}^*}^2 \snorm{\TT_{\vo}}^2}{(\inf_{\la \in \La}  \ka_\la)^{2}} \norm{\data}^2,
\end{multline*}
which implies  $\Ko^\plus$ is bounded.

 \ref{thm:char2}
Let $\Ko^\plus$ be bounded with norm $\snorm{\Ko^\plus}$ and suppose  $\inf_{\la \in \La}  \ka_\la = 0$.
Then the family $(\ka_\la^{-1} v_\la)_{\la \in \La}$ has no upper frame bound.
This can be shown by contradiction: Suppose it has an upper frame bound $B$ we know that $\sup_{\la \in \La}\snorm{\ka_\la^{-1} v_\la} \leq \sqrt{B}$, but since $\vo$ is norm bounded from below we have $\sup_{\la \in \La}\snorm{\ka_\la^{-1} v_\la}=\infty$.
Hence we have that for all constants $B > 0$ there exists $y \in \overline{\ran \Ko}$ such that
\begin{equation} \label{contr}
\sum_{\la \in \La} \abs{\inner{y}{\ka_\la^{-1} v_\la}}^2 > B \norm{\data}^2.
\end{equation}
Now choose $B=\snorm{\TT_{\buo}^\plus}^2 \snorm{\Ko^\plus}^2$, where $\buo$ is an arbitrary dual frame of $\uo$, and let  $y$ be such that \eqref{contr} is satisfied.
It is well known that if $\Ko^\plus$ is bounded, $\Ko$ has closed range \cite{groetsch2006stable}. Thereby, $y \in \dom (\Ko^\plus)$. Moreover, it has the unique representation $y=\Ko \signal^\plus$ with $\signal^\plus \in \ker(\Ko)^\perp\cap \dom(\Ko)$ and by $\inner{\Ko \signal^\plus}{\ka_\la^{-1} v_\la} = \inner{\signal^\plus}{u_\la}$ follows that $(\inner{y}{\ka_\la^{-1} v_\la})_{\la \in \La} \in \ell^2(\La)$.
Then we have
\begin{multline*}
\norm{\Ko^\plus y}^2 = \norm{\sum_{\la \in \La} \ka_\la^{-1}\inner{y}{ v_\la} \bar u_\la}^2
\geq \frac{1}{\snorm{\TT_{\buo}^\plus}^2} \sum_{\la \in \La} \abs{\inner{y}{\ka_\la^{-1} v_\la}}^2 \\
> \frac{1}{\snorm{\TT_{\buo}^\plus}^2} B \norm{\data}^2 = \snorm{\Ko^\plus}^2 \norm{\data}^2 \,, 
\end{multline*}
which leads to a contradiction.
\end{proof}

Compact operators with infinite dimensional range are typical examples of linear operators with non-closed range. Moreover, the spectral theorem for compact operators states that zero is the only accumulation point of the singular values $(\sigma_\la)_{\la \in \La}$. This means that we can find a bijection $\pi \colon  \N \to \La$ such that $(\ka_{\pi(n)})_{n \in \N}$ is a decreasing null-sequence. Below we show that the same holds for a DFD if $\uo$ is norm bounded from below. 

\begin{theorem}[Quasi-singular values for compact operators]
Suppose that $\Ko\colon \X \to \Y$ is a compact linear operator and assume that $(\uo, \vo, \kao)$ is a DFD for $\Ko$, where $\uo$ is norm bounded from below. Then, zero is the only accumulation point of $\kao$.
\end{theorem}

\begin{proof}
Without loss of generality  consider the case $\La = \N$.
 Aiming for a contradiction,    we assume that  $\kao$ has an  accumulation point different from zero ($\infty$ is allowed). Therefore we can find a subsequence $(\ka_{n(k)})_{k \in \N}$ with  $\inf_{k \in \N} \ka_{n(k)} := c >0$.
Consequently $\snorm{v_{n(k)} / \ka_{n(k)} } \leq c^{-1} \sqrt{B_\vo}$, where $B_\vo$ is the upper frame bound of $\vo$. In particular, the sequence $(v_{n(k)}/\ka_{n(k)})_{k\in \N}$ is bounded.
Because  $\Ko^*$ is  compact, there exists another subsequence $(v_{n(k(\ell))}/\ka_{n(k(\ell))})_{k \in \N}$ such that $u_{n(k(\ell))} = \Ko^* (v_{n(k(\ell))}/\ka_{n(k(\ell))})$ strongly converges  to some  $x \in \ran (\Ko)^* \subseteq \ker(\Ko)^\perp$. Because  $\uo$ is norm bounded from below we have  $x \neq 0$. Choose $\eps > 0 $ such that $\snorm{x}^2 \geq 2\eps$. Since $u_{n(k(\ell))} \to x$ we can choose $N \in \N$ such that $\forall \ell \geq N \colon  \snorm{u_{n(k(\ell))}-x}^2 < \eps$. From this it follows $2\operatorname{Re}(\inner{u_{n(k(\ell))}}{x}) > \snorm{u_{n(k(\ell))}}^2 + \norm{x}^2 - \eps >  \eps$.
Consequently,
\begin{align*}
\sum_{n \in \N} \abs{\inner{\signal}{u_n}}^2 \geq \sum_{\ell = N}^\infty \abs{\inner{\signal}{u_{n(k(\ell))}}}^2 \geq \sum_{\ell = N}^\infty \frac{\eps^2}{4} = \infty.
\end{align*}
This contradicts the frame condition of $\uo$.
\end{proof}

If $\uo$ is not norm bounded from below, $(\ka_\la)_{\la \in \La}$ can have one or more accumulation points as the following elementary example shows. Note that this example is not intended as representative forward operator we are interested in, but rather indicates to be careful when the frames are not bounded from below.

\begin{example}\label{ex:HP}
Let $\X = \Y = \ell^2(\N)$ and consider the diagonal multiplication operator  $
\Ko \colon \ell^2(\N) \to \ell^2(\N) \colon  (x_i)_{i \in \N} \mapsto \left(x_i/\sqrt{i+1}\right)_{i \in \N}$.
Clearly $\Ko$ is  self-adjoint and compact with SVD given
by $( (e_i)_{i\in \N},   (e_i)_{i\in \N}, (1/\sqrt{i+1})_{i\in \N})$ where $(e_i)_{i\in \N}$ denotes the  standard basis of $\ell^2(\N)$.
Define
\begin{align*}
\uo & \coloneqq \Bigl( e_0, e_0 \mid \frac{e_1}{\sqrt{2}},\frac{e_1}{\sqrt{2}},\frac{e_1}{\sqrt{2}} \mid \frac{e_2}{\sqrt{3}},\frac{e_2}{\sqrt{3}},\frac{e_2}{\sqrt{3}},\frac{e_2}{\sqrt{3}} \mid \ldots \Bigr)
 \\
\vo & \coloneqq \Bigl( e_0, e_0  \mid e_1,\frac{e_1}{\sqrt{2}},\frac{e_1}{\sqrt{2}} \mid ,e_2, \frac{e_2}{\sqrt{3}},\frac{e_2}{\sqrt{3}},\frac{e_2}{\sqrt{3}} \mid  \ldots \Bigr)
\\
\kao& \coloneqq  \Bigl(  1, 1 \mid 1, \frac{1}{\sqrt{2}},\frac{1}{\sqrt{2}} \mid  1, \frac{1}{\sqrt{3}},\frac{1}{\sqrt{3}},\frac{1}{\sqrt{3}} \mid \ldots \Bigr)  \,.
\end{align*}
For $x \in \ker(\Ko)^\perp = \X$ and    $y \in \overline{\ran(\Ko)} = \Y$ we have
\begin{align*}
\sum_{\la \in \La} \abs{\inner{\signal}{u_\la}}^2
&=  \sum_{n \in \N} (n+2) \abs{\inner{\signal}{\frac{e_n}{\sqrt{n+1}}}}^2
\\&= \sum_{n \in \N} \abs{\inner{\signal}{e_n}}^2  + \sum_{n \in \N} \frac{1}{n+1} \abs{\inner{\signal}{e_n}}^2
\\
\sum_{\la \in \La} \abs{\inner{y}{v_\la}}^2 &= \sum_{n \in \N} \abs{\inner{y}{e_n}}^2 + \sum_{n \in \N} (n+1) \abs{\inner{y}{\frac{e_n}{\sqrt{n+1}}}}^2
\\& = \norm{\data}^2 + \norm{\data}^2 = 2 \norm{\data}^2 \,.
\end{align*}

Hence  $\uo$ is a frame  with frame bounds $A=1$ and  $B=2$ and $\vo$ is a frame with bounds $A=B=2$. Moreover, the quasi-singular value relation $\Ko^* v_\la  =  \kappa_\la u_\la$ holds. Therefore  $(\uo,\vo,\kao)$ is a DFD for the compact operator  $\Ko$. However, the sequence  $\kao$ has accumulation points $0$ and $1$.
\end{example}

Note  that we can easily modify example \ref{ex:HP}  such that $\infty$ is an accumulation point of $\kao$. To see this consider $\Ko$ and $\vo$ from the example above and change $\uo$ and $\kao$ to
\begin{align*}
\uo &=\Bigl(e_0,e_0 \mid \frac{e_1}{2},\frac{e_1}{\sqrt{2}},\frac{e_1}{\sqrt{2}} \mid\frac{e_2}{3},\frac{e_2}{\sqrt{3}},\frac{e_2}{\sqrt{3}},\frac{e_2}{\sqrt{3}} \mid \ldots \Bigr)
\\
\kao &= \Bigl( 1,1 \mid \sqrt{2}, \frac{1}{\sqrt{2}},\frac{1}{\sqrt{2}} \mid \sqrt{3}, \frac{1}{\sqrt{3}},\frac{1}{\sqrt{3}},\frac{1}{\sqrt{3}} \mid \ldots \Bigr).
\end{align*}
Then  $(\uo,\vo ,\kao)$ is still a valid DFD of $\Ko$
where $\uo$ has  frame bounds $A=1$ and $B=2$,
and $\kao$ has  accumulation points $0$ and $\infty$.

\section{Regularization by filtered DFD}
\label{sec:reg}

Throughout this section, let $(\uo, \vo, \kao)$ be a DFD of the operator $\Ko \colon \dom (\Ko ) \subseteq  \X \to \Y$ and $\buo$ a dual frame of $\uo$.  Recall that we allow the forward operator  $\Ko$ to be unbounded.  For typical inverse problems, the Moore Penrose inverse $\Ko^\plus$ is unbounded and has to be regularized. In this section we develop a regularization concept by filtered DFDs.  

\subsection{Filtered DFD}

A wide class of classical regularization methods can be constructed by spectral filtering.  Below we extend these concepts to regularization by filtering a DFD. We start by defining regularizing filters using properties similar to  \cite[Theorem 4.2]{engl1996regularization}. Be aware that our filter functions $f_\al$ correspond  to $\kappa g_\al(\kappa^2)$ where $g_\al$ are the filter functions  used in  \cite{engl1996regularization}.

\begin{definition}[Regularizing filter]\label{def:refF}
We call a  family $(f_\al)_{\al>0}$ of piecewise continuous functions $f_\al\colon (0, \infty)  \to \R$  a regularizing filter if,
\begin{enumerate}[itemindent =2em, leftmargin =2em, label=(F\arabic*)]
\item\label{F1}  $\forall \al > 0 \colon 
 \norm{f_\al}_\infty  < \infty$.
\item\label{F2} $\exists C>0 \colon  \sup\{  \abs{ \kappa f_\al(\kappa )} \mid \al > 0 \wedge  \kappa \geq 0 \}\leq C$.
\item\label{F3} $\forall \kappa \in (0,\infty) \colon \lim_{\al\to 0} f_\al(\kappa )=1/\kappa $.
\end{enumerate}
\end{definition}

Using a regularizing filter we define the following central concept of this paper.

\begin{definition}[Filtered DFD] \label{def:fdfd}
Let $(f_\al)_{\al>0}$ be a regularizing filter and define
\begin{equation} \label{eq:fex}
\forall \al >0 \colon \quad
\Freg_\al \colon \Y \to \X \colon
\data \mapsto  \sum_{\la \in \La} f_\al(\ka_\la)\inner{y}{v_\la}\bar u_\la \,.
\end{equation}
We call the family  $(\Freg_\al)_{\al >0}$ the filtered diagonal frame decomposition (filtered DFD)  according to $(f_\al)_{\al>0}$ based on the DFD $(\uo, \vo, \kao)$ and the dual frame  $\buo$.
\end{definition}

As  mentioned above,  our filter functions $f_\al$ correspond  to $\kappa g_\al(\kappa^2)$ where $g_\al$ are the filters  commonly used  in  spectral filtering. In spectral filtering, the operators $ g_\al( \Ko^*\Ko) \Ko^* (\data) = \sum_{\la \in \La} \ka_\la g_\al(\ka_\la^2) \inner{\data}{u_\la} \bar u_\la $ are derived from functional calculus. Further note that our assumptions \ref{F1}-\ref{F3} with $f_\al(\kappa)  = \kappa g_\al(\kappa^2)$ are weaker than the ones  \cite[Theorem 4.2]{engl1996regularization} where $g_\al$ is assumed to be bounded.

\subsection{Convergent regularization methods}

Below we show that filtered DFD yields a well defined convergent regularization method. To that end, we recall the definition of a regularization method taken from   \cite[Definition~3.1]{engl1996regularization} for case of bounded $\Ko$ and adopted to the unbounded case considered here. For regularization with unbounded forward operators see, for example,  \cite{hofmann2009regularization,groetsch2006stable}.

\begin{definition}[Regularization method]
Let $(\Ro_\al)_{\al>0}$ be a family of continuous operators $\Ro_\al \colon \Y\to \X$, $\data  \in  \dom(\Ko^\plus)$ and $\al^\ast\colon (0,\infty)\times \Y \to (0,\infty)$. Then the pair $((\Ro_\al)_{\al>0},\al^\ast)$ is a regularization method for the solution of $\Ko \signal =\data$, if
\begin{align*}
&\lim_{\delta\to 0} \sup\{\al^\ast(\delta,\data^\delta)\mid \data^\delta\in \Y \wedge   \|\data^\delta - y\|\leq\delta\}=0
\\
&\lim_{\delta\to 0} \sup\{\| \Ko^\plus y -\Ro_{\al^\ast(\delta,\data^\delta)}\data^\delta \| \mid \data^\delta \in \Y \wedge     \|\data^\delta-\data\|\leq\delta\} = 0 \,.
\end{align*}
In this case we call $\al^\ast $ an admissible parameter choice. If for any $\data  \in  \dom(\Ko^\plus)$ there exists an admissible parameter choice, then  we call  $(\Ro_\al)_{\al>0}$ a regularization of  the Moore Penrose inverse $\Ko^\plus$.
\end{definition}

Given an SVD $(u_n, v_n, \sigma_n)_{n\in \N}$ of $\Ko$ and a regularizing filter $(f_\al)_{\al>0}$, it is well known (at least if $\kappa \mapsto \kappa^{-1} f_\al (\kappa)$ is  bounded) that  the family
\begin{equation*}
 \sum_{n \in \N}
g_\al(\sigma_n^2) \inner{\Ko^\ast  y}{u_n} u_n =
\sum_{n \in \N}
 f_\al (\sigma_n) \inner{y}{v_n} u_n
=  \Freg_\al (y)
\end{equation*}
with $f_\al (\sigma_n)=  \sigma_n g_\al (\sigma_n^2)$ defines a regularization method \cite[Theorem~8]{engl1996regularization} together with convergence rates.
Two prominent examples of filter-based regularization methods are classical Tikhonov regularization and truncated SVD.
In truncated SVD,  the regularizing  filter is given by
$f_\al(\sigma) = \sigma^{-1} \chi_{ [\al, \infty) } (\sigma^2)$.
In Tikhonov regularization, the regularizing filter is given by  $f_\al(\sigma)=\sigma/(\sigma^2 +\al)$.  In this paper we generalize such  results by allowing a DFD instead of the SVD. To that  end we use the following well known result.

\begin{lemma}[Characterization of linear regularizations] \label{lem:apriori}
Let   $\skl{\Ro_\al}_{\al >0}$ be a family of linear bounded operators  which pointwise converge
to $\Ko^\plus$ on $\dom \skl{\Ko^\plus}$ and let  $y \in \dom(\Ko^\plus)$. If  the parameter choice $\al^\ast \colon \skl{0, \infty} \to \skl{0, \infty}$ satsfies $\lim_{\delta \to 0}  \al^\ast\skl{\delta} = \lim_{\delta \to 0} \delta \snorm{\Ro_{\al^\ast}\skl{\delta}} =0$, then the pair
$\skl{ \skl{\Ro_\al}_{\al >0}, \al^\ast}$ is a regularization
method for  $\Ko x =y$.\end{lemma}

\begin{proof}
For the case of bounded forward operators see, for example,  \cite[Proposition 3.7]{engl1996regularization}. The simple proof  is based  on the estimate $\norm{\Ko^\plus \data - \Ro_\al \data^\delta} \leq \norm{\Ko^\plus \data - \Ro_\al \data}  + \delta \norm{\Ro_\al} $ and applies to  case of unbounded $\Ko$.
\end{proof}

\subsection{Well-posedness and convergence}

Let $(f_\al)_{\al>0}$ be a regularizing filter and $(\Freg_\al)_{\al >0}$ be the filtered DFD defined by \eqref{eq:fex}.

\begin{proposition}[Existence and stability] \label{prop:well}
For any $\al>0$ the operator $\Freg_\al$ is well defined, linear and bounded.  Moreover, $\norm{\Freg_\al} \leq \norm{ f_\al}_\infty   (B_{\buo} B_\vo)^{1/2}$, where $B_{\buo}$ and $B_\vo$ are the upper frame bounds of $\buo$ and $\vo$, respectively.
\end{proposition}

\begin{proof} 
Let $\al > 0$, $\data \in \Y$. According to \ref{F1},  $f_\al$ is bounded and therefore   $(f_\al(\ka_\la)\inner{y}{v_\la})_{\la \in \La} \in \ell^2(\La)$. Further,
$\norm{\Freg_\al \data}^2 
=\| \sum_{\la \in \La} f_\al(\ka_\la) \inner{y}{v_\la}\bar u_\la \|^2
 \leq  \norm{f_\al}_\infty^2 B_{\buo} B_\vo \norm{\data}^2$ which shows that $\Freg_\al \data$ is well defined and bounded with $\norm{\Freg_\al} \leq \norm{ f_\al}_\infty   (B_{\buo} B_\vo)^{1/2}$.
\end{proof}

\begin{proposition}[Pointwise convergence]\label{prop:point}
For all $\data \in \dom(\Ko^\plus) \colon $ $\lim_{\al \to 0} \Freg_\al \data = \Ko^\plus \data$.
\end{proposition}

\begin{proof}
Let $\data = \data^\plus + \data^\perp \in \ran(\Ko) \oplus \ran(\Ko)^\perp$ and set $\signal^\plus := \Ko^\plus \data  \in \ker(\Ko)^\perp\cap \dom(\Ko)$. Then $\Ko \signal^\plus = \Po_{\overline{\ran(\Ko)}} y  = \data^\plus$ and therefore
\begin{align*}
\norm{\signal^\plus - \Freg_\al \data}^2
&= \hnorm{\signal^\plus - \sum_{\la \in \La} f_\al(\ka_\la)\inner{y}{v_\la}\bar u_\la}^2 \\
&=\hnorm{\signal^\plus - \sum_{\la \in \La} f_\al(\ka_\la)\inner{\data^\plus}{v_\la}\bar u_\la}^2
\\
&
= \hnorm{\signal^\plus - \sum_{\la \in \La} f_\al(\ka_\la)\inner{\Ko \signal^\plus}{v_\la}\bar u_\la}^2
\\
&= \hnorm{\sum_{\la \in \La} \inner{\signal^\plus}{u_\la} \bar u_\la - \sum_{\la \in \La} \ka_\la f_\al(\ka_\la)\inner{\signal^\plus}{u_\la}\bar u_\la}^2
\\
& = \hnorm{\sum_{\la \in \La} (1 - \ka_\la f_\al(\ka_\la)) \inner{\signal^\plus}{u_\la}\bar u_\la}^2 \\
& \leq B_{\buo}  \sum_{\la \in \La} \abs{1-\ka_\la f_\al(\ka_\la)}^2 \abs{\inner{\signal^\plus}{u_\la}}^2
\\
& \leq \sup_{\la \in \La} \abs{1-\ka_\la f_\al(\ka_\la)}^2 B_{\buo} B_\uo \norm{\signal^\plus}^2 \,.
\end{align*}
According to \ref{F2}, \ref{F3}   we have  $\sup_{\alpha},\la \abs{1-\ka_\la f_\al(\ka_\la)}^2 < \infty$ and $\lim_{\al \to 0} \abs{1- \ka_\la f_\al(\ka_\la)} =0$ pointwise.
Therefore, application of the dominated convergence theorem to the series in the second last line yields $\snorm{\signal^\plus - \Freg_\al \data}^2 \to 0$ for $\al \to 0$.
\end{proof}

By collecting the above results we obtain the  following convergence theorem for filtered DFD.

\begin{theorem}[Convergence]\label{thm:conv}
Let $(f_\al)_{\al>0}$ be a regularizing filter, $(\uo, \vo, \kao)$ be a DFD of  $\Ko \colon \dom(\Ko) \subseteq \X \to \Y$ and $\buo$ a dual frame of $\uo$.
Then $((\Freg_\al)_{\al >0},\al^*)$ is a regularization method for $\Ko x=y$ provided  that the parameter choice $\al^* \colon (0, \infty) \to (0,\infty)$
satisfies $0 =\lim_{\delta \to 0} \al^*(\delta)= \lim_{\delta \to 0} \delta  \snorm{f_{\al^*(\delta)}}_\infty $.
\end{theorem}

\begin{proof}
According to Propositions \ref{prop:well} and \ref{prop:point},  $(\Freg_\al)_{\al >0}$ is a family  of bounded linear operators that converges  pointwise to $\Ko^\plus$ on $\dom(\Ko)$.
According to Lemma \ref{lem:apriori}  the pair $((\Freg_\al)_{\al >0},\al^*)$ is a regularization method if $\al^\star(\delta), \delta \snorm{\Freg_{\al^*(\delta)}} \to 0$ as $\delta \to 0$. The estimate  $\norm{\Freg_\al} \leq \norm{f_\al}_\infty \sqrt{B_{\buo} B_\vo}$ of Proposition~\ref{prop:well}  finally yields the claim.
\end{proof}

\subsection{Convergence rates}

Next we derive convergence rates which give quantitative estimates on the
reconstruction error $\norm{\signal^\plus - \signal_\al^\delta}$.

\begin{theorem}[Convergence rates]\label{thm:rates}
Let $(f_\al)_{\al>0}$ be a regularizing filter, $(\uo, \vo, \kao)$ be a DFD of $\Ko$,  $\buo$ a dual frame of $\uo$ and  $(\Freg_\al)_{\al >0}$ be the filtered DFD defined by \eqref{eq:fex}. For given numbers $\rho, \mu > 0$ and some constant $C_\mu$ suppose
\begin{enumerate}[itemindent =2em, leftmargin =2em,label=(R\arabic*)]
\item\label{R1} $\norm{f_\al}_\infty = \mathcal{O}({\al^{-1/2})}$ as $\al \to 0$,
\item\label{R2} $\forall \al >0 \colon  \sup\{{\ka^{2\mu}  \abs{1-\ka f_\al(\ka)}} \mid \ka \in (0, \infty)\} \leq  C_\mu \al^{\mu}$,
\item\label{R3} $\al = \al^*(\delta,\data^\delta) \asymp \left(\delta / \rho \right)^{2 /(2\mu + 1)}$.
\end{enumerate}
Suppose  $\signal^\plus \in \X$ satisfies  the following source-type condition
\begin{equation}  \label{eq:sc}
\exists \bom  \in \ell^2(\La) \colon
\Bigl( \norm{\bom}_2 \leq \rho
 \;  \wedge \;   \forall \la \in \La \colon \inner{\signal^\plus}{u_\la}=\ka_\la^{2\mu} \omega_\la 
 \Bigr)
  \,.
\end{equation}
Then, for   some constant $c=c_\mu$ and all  $\data^\delta \in \Y $ with  $\norm{\data^\delta-\Ko \signal^\plus} \leq \delta$ with sufficiently small $\delta$, the following convergence rate result holds:
\begin{equation*}
 \norm{\signal^\plus - \Freg_{\al^*}(\data^\delta)}  \leq c_\mu \,
 \delta^\frac{2\mu}{2\mu+1}  \rho^\frac{1}{2\mu+1}  \,.
\end{equation*}
\end{theorem}

\begin{proof}
Let $\signal^\plus$, $\bom$, $\data^\delta$ satisfy $\inner{\signal^\plus}{u_\la}=\ka_\la^{ 2 \mu} \omega_\la$, $\norm{\bom}_{\ell^2} \leq \rho$, $\norm{\data^\delta-\Ko \signal^\plus} \leq \delta$. Then
\begin{align*}
\norm{\Freg_\al(\data^\delta) - \signal^\plus} & \leq \norm{\Freg_\al(\data^\delta - \Ko \signal^\plus)} + \norm{\Freg_\al(\Ko \signal^\plus) - \signal^\plus}\\
& \leq \norm{\Freg_\al}\delta + \hnorm{\sum_{\la \in \La} (1 - \ka_\la f_\al(\ka_\la)) \inner{\signal^\plus}{u_\la}\bar u_\la}\\
& \leq  \sqrt{B_{\buo} B_\vo  } \,  \norm{f_\al}_\infty  \delta +  \biggl(  B_{\buo} \sum_{\la \in \La} \abs{1-\ka_\la f_\al(\ka_\la)}^2 \abs{\inner{\signal^\plus}{u_\la}}^2  \biggr)^\frac{1}{2} \\
& \leq c_1 \al^{ -1/2} \, \delta +  \biggl( B_{\buo} \sum_{\la \in \La} \abs{1-\ka_\la f_\al(\ka_\la)}^2 \abs{\ka_\la^{ 2\mu} \omega_\la}^2 \biggr)^\frac{1}{2}\\
& \leq c_1 \al^{-1/2} \, \delta + \sqrt{B_{\buo}}C_\mu \al^\mu \rho.
\end{align*}
Now choose $\al = \al^*(\delta,\data^\delta) \asymp ( \delta / \rho )^{ 2 / (2\mu + 1)}$. Then the above estimate implies
\begin{equation*}
\norm{\Freg_{\al^*(\delta,\data^\delta)}(\data^\delta) - \signal^\plus} \leq c_2 \left( \delta^{ 1-\frac{1}{2\mu +1}} \rho^{  \frac{1}{2\mu +1}} + \delta^{ \frac{2\mu}{2\mu +1}} \rho^{ 1-\frac{2\mu}{2\mu +1}} \right) = \mathcal{O}\left(\delta^{ 2\mu/(2\mu +1)} \right) \,,
\end{equation*}
and completes the proof.
\end{proof}

\begin{remark}[Qualification of a filter] \label{rem:qualification}For a given regularizing filter, the condition \ref{R2} may only hold for  $ \mu \in (0, \mu_0]$ but not for $\mu  > \mu_0$.  The index  $ \mu_0$  is often called the qualification of the regularizing filter (see the discussion on \cite[page 76]{engl1996regularization}).  If \ref{R2} holds for all $\mu >0$, the qualification is said to be infinite. It is known  that the qualification of  $f_\al(\kappa) = \kappa/(\kappa^2 + \alpha) $ is $\mu_0 = 1$ and that $f_\al   (\kappa) = \kappa^{-1} \chi_{[\alpha, \infty)}(\kappa^2) $ has infinite qualification.  
\end{remark}

\begin{remark}[Source conditions and generalizations]  
In Theorem~\ref{thm:rates}  we derived convergence rates for elements satisfying the source-type condition \eqref{eq:sc} which can be written as $\TT_{\uo} \signal^\plus \in \ran( \M_\kao^{2\mu} )$. This may be seen as an abstract  smoothness condition for $\signal^\plus$. As in the case of classical  spectral filtering one could  study source conditions of the form $\TT_{\uo} \signal^\plus \in \ran( \phi(\M_{\kao}) )$ for more general index functions $\phi$. For example, logarithmic source conditions are useful for exponentially ill-posed problems; see \cite{hohage2000regularization}. Another generalization  is the use of approximate source conditions based on distance functions \cite{hofmann2006approximate}. Investing such concepts in the context of DFDs seems very interesting but beyond the scope of the present article.  
\end{remark}

\subsection{Examples of regularizing filters} \label{sec:regfil}

In this subsection we study  examples of filtered DFDs, namely  truncated DFD and  Tikhonov-filtered DFD.  We verify that the corresponding  filters satisfy the requirement for being  regularizing and also  that the convergence rate conditions  in  Theorem \ref{thm:rates} are satisfied for all  $\mu >0$ in case of truncated DFD and for $\mu \leq 1$ in case of Tikhonov-filtered DFD.  

\paragraph{Truncated DFD:}

For any $\al >0$ consider the cut-off functions 
\begin{equation*}
f^{(1)}_\al \colon (0, \infty) \to \R \colon \kappa \mapsto  
 \begin{cases}
1/\ka \quad &\text{if } \ka^2 \geq \al \\
0 \quad &\text{if } \ka^2  <  \al \,.
\end{cases}
\end{equation*}
Obviously conditions \ref{F1}-\ref{F3} in Definition~\ref{def:refF} are satisfied with $C=1$ which implies that  $(f^{(1)}_\al)_{\al > 0}$ is a regularizing filter.  Furthermore,  $\sup\{\ka^{2\mu} \abs{1-\ka f^{(1)}_\al(\ka)} \mid \ka  >0\} = \sup\{\ka^{2\mu} \abs{1-\ka f^{(1)}_\al(\ka)} \mid \ka^2 < \al \} = \al^\mu$ for all  $\al, \mu > 0$.  Hence the convergence rates conditions  \ref{R1}, \ref{R2}  of Theorem~\ref{thm:rates} are satisfied.  The corresponding  filtered DFD becomes  
\begin{equation} \label{eq:TDFD}
	\Freg^{(1)}_\al(y) \coloneqq \sum_{\ka_\la^2 \geq \al} \frac{1}{\ka_\la} \inner{y}{v_\la} \bar u_\la \,.
\end{equation}
In the special case where $(\uo,\vo,\kao)$  is an  SVD for  $\Ko$, this is well-known truncated SVD. In analogy, for general DFDs we name  \eqref{eq:TDFD} truncated DFD.

The considerations above allow application of Propositions~\ref{prop:well}, \ref{prop:point}  and Theorem~\ref{thm:conv} showing well-posedness, stability and convergence of \eqref{eq:TDFD}.  Moreover Theorem~\ref{thm:rates} can be applied for any $\mu>0$.  
 Thus for $\signal^\plus$  with $\TT_\uo \signal^\plus \in \ran ( \M_\kappa^{2\mu})$,  the parameter  choice $\al    \asymp \delta^{2/(2\mu + 1)}$  yields the convergence rate $ \snorm{\signal^\plus - \Freg^{(1)}_\al} \data^\delta  = \mathcal{O}  (\delta^{2\mu/(2\mu+1)})$.

\paragraph{Tikhonov type DFD:} \label{anotherfilter}

For any $\al >0$ consider the Tikhonov filter
\begin{equation*}
f^{(2)}_\al \colon (0, \infty) \to \R \colon \kappa \mapsto  \frac{\ka}{\ka^2+\al} \,.
\end{equation*} 
For all $\kappa, \al > 0$ we have 
$ \lvert \kappa  f^{(2)}_\al(\kappa )  \rvert = \lvert \ka^2/(\ka^2 + \al) \rvert   \leq 1$
and  $\lim_{\al \rightarrow 0} f^{(2)}_\al (\kappa) = 1/\ka$. Further, $ f^{(2)}_\al$     is bounded, takes its maximum at $\ka^2 = \al$ and  $\norm{f^{(2)}_\al}_\infty =  \al^{-1/2}/2$. Hence  conditions \ref{F1}-\ref{F3} are satisfied and $(f^{(2)}_\al)_{\al >0}$ is a regularizing filter in the sense of  Definition~\ref{def:refF}.  
Moreover, for $\mu \in (0,1]$ the function $\kappa \mapsto \ka^{2\mu} \abs{1 -\ka f^{(2)}_\al (\kappa)}  = \ka^{2\mu} \al / (\ka^2 + \al)$ has its supremum at $\ka =  \sqrt{\al \mu/(1-\mu)}$. Thus $\sup_\ka  \ka^{2\mu} \abs{1 -\ka f^{(2)}_\al (\kappa)}  = \left( (2\mu)^{\mu} (2-2\mu)^{1-\mu} \right)/2 \al^\mu $.  Hence the filter  $(f^{(2)}_\al)_{\al >0}$ also satisfies the convergence rates conditions of Theorem \ref{thm:rates}.

According to the  above considerations,  Propositions~\ref{prop:well}, \ref{prop:point}  and Theorem~\ref{thm:conv} show  well-posendess and convergence of the filtered DFD    
\begin{equation} \label{eq:zikDFD}
	\Freg^{(2)}_\al( y ) \coloneqq \sum_{\la \in \La}    \frac{\ka_\la}{\ka_\la^2+\al}  \inner{y}{v_\la} \bar u_\la.
\end{equation}
Moreover,   for $\mu \in (0,1]$,  the parameter  choice $\al   \asymp \delta^{2/(2\mu + 1)}$  yields the convergence rate $ \lVert \signal^\plus - \Freg_\al \data^\delta \rVert  = \mathcal{O}
 (\delta^{2\mu/(2\mu+1)})$.  In the special case where  $(\uo,\vo,\kao)$  is a SVD then \eqref{eq:zikDFD} reduces to Tikhonov regularization as in this case $\Freg^{(2)}_\al \data $ equals the minimizer  of the Tikhonov functional $\norm{\Ko x - y}^2 + \alpha \norm{x}^2$. For general DFDs this relation does not hold true. 
 
Notice that the Tikhonov filter $(f^{(2)}_\al)_{\al >0}$ does not satisfy \ref{R2} for  $\mu>1$, which means that  the Tikhonov filter has qualification  $\mu_0=1$; see Remark~\ref{rem:qualification}. This is one motivation for considering regularization methods with higher qualification that can also be implemented without knowledge  of the SVD, such as iterated Tikhonov regularization. Anyway, in this work we allow more general DFDs  which provides an alternative  strategy  to avoid numerically costly SVD computation.

\subsection{Order optimality}

In the following we prove that  the  convergence rates obtained in Theorem~\ref{thm:rates} are order optimal for the  source set defined  by \eqref{eq:sc} in the special case that the  frame $\uo$ admits a biorthogonal sequence $\buo = (u_\la)_{\la\in\La}$ with $\forall \la,  \nu  \in \La \colon  
	\inner{u_\la}{\bar{u}_\nu}=\delta_{\la \nu} $.
The requirement that $\uo$ has a biorthogonal sequence is equivalent to $\uo$ being a Riesz-basis of $\ker(\Ko)^\perp$. 
To do this, we define 
\begin{equation} \label{eq:sourceset}
{ U_{\mu,\rho}}  \coloneqq \Bigl\{ x\in\dom(\Ko) \mid \inner{x}{u_{\lambda}}=\kappa_{\lambda}^{2 \mu}w_{\lambda} \wedge  \sum_{ \la \in \La}|w_{\lambda}|^{2} = \rho^{2} \Bigr\}
\end{equation}
 and for any set $\mathcal{M}  \subseteq \dom(\Ko)$ define 
$ \epsilon(\mathcal{M},\delta)  \coloneqq  \sup\{\norm{x}\mid x\in \mathcal{M} \wedge \norm{\Ko x}\leq\delta\}$.

We have that $\epsilon(\mathcal{M}, \delta)$ is a lower bound for the worst case reconstruction error  \begin{equation} \label{eq:worst}
	   E(\mathcal{M},\delta, \Ro)
	   \coloneqq 
	   \sup\{ \snorm{\Ro \data-x} \mid x\in \mathcal{M} \wedge \data^\delta \in \Y \wedge \norm{\Ko x -\data^\delta}\leq \delta\}  \,,
\end{equation}
for an arbitrary mapping  $\Ro \colon \Y \to \X$ (in this context called reconstruction method) with $\Ro(0)  = 0$; see  \cite{engl1996regularization}.  A family $(\Ro^{\delta})_{\delta>0}$ of reconstruction methods is called order optimal on $\mathcal{M}$, if  
 $E(\mathcal{M},\delta, \Ro^{\delta}) \leq c \, \epsilon(\mathcal{M}, \delta)$   for all sufficiently small $\delta$ and some constant $c>0$. To show that the convergence rate of  Theorem~\ref{thm:rates} is order optimal therefore amounts to bound $ \epsilon(U_{\mu,\rho},\delta)$. 

\begin{theorem}\label{thm:opt}
Let $(\uo, \vo, \kao)$ be a DFD of $\Ko$ such that $\uo$ has a biorthogonal sequence $\buo$  and $0$ is an accumulation point of $\kao$. Then for the source sets $U_{\mu,\rho}$ defined by \eqref{eq:sourceset} and some sequence $(\delta_n)_{n\in\N}$ converging to $0$, we have
  \begin{equation*}
    { \epsilon(U_{\mu,\rho},\delta_n})  \geq \sqrt{\frac{B_{\vo}}{A_{\uo}}}   \,  \delta_n^{\tfrac{2  \mu}{ 2 \mu+1}}\rho^{\tfrac1{ 2 \mu+1}} \,.
  \end{equation*}
In particular, under the  assumptions of Theorem~\ref{thm:rates}, the  family $(\Freg_{\al^*(\delta,  \cdot)})_{\delta >0}$ is an order optimal reconstruction method for  the source set  $U_{\mu,\rho}$.
\end{theorem}

\begin{proof}
After extracting a subsequence we assume without loss of generality that $\Lambda = \N$ and that $\kao$ converges to $0$. For any $\nu\in\N$ set $x_{\nu}\coloneqq \rho \kappa_\nu^{ 2 \mu} \bar{u}_\nu$ such that
\[
  \scp{x_{\nu}}{u_{\lambda}} = \kappa_{\lambda}^{2 \mu}w_{\lambda},\quad w_{\lambda} =
  \begin{cases}
    \rho,\ & \text{if}\ \lambda=\nu\\
    0, & \text{else} \,.
  \end{cases}
\]
By definition we have $\norm{w}_{2} = \rho$ and $x_\nu \in U_{\mu,\rho}$.
If we consider the decreasing null-sequence of noise levels $\delta_\nu = \rho\kappa_{\nu}^{2 \mu+1} / \sqrt{A_\vo}$ we get
\begin{align*}
  \norm{x_\nu}^{2} & \geq \frac{1}{B_{\uo}}\sum_{\lambda\in\La}|\scp{u_{\lambda}}{x_\nu}|^{2} = \frac{1}{B_{\uo}}\kappa_{\nu}^{4 \mu}\rho^{2} = A_\vo^{2 \mu/(2\mu+1)} \frac{1}{B_{\uo}}\left(\delta_\nu^{2  \mu/(2\mu+1)}\rho^{1/({2\mu}+1)}\right)^{2}
\end{align*}
and
\begin{equation*}
  \norm{\Ko x_\nu}^{2} \leq \frac{1}{A_{\vo}}\sum_{\lambda\in\La}|\scp{v_{\lambda}}{\Ko x_\nu}|^{2}
                =  \frac{1}{A_{\vo}}\sum_{\lambda\in\La}\kappa_{\lambda}^{2}|\scp{u_{\lambda}}{x_\nu}|^{2}
                = \frac{1}{A_{\vo}}\kappa_{\nu}^{2({2\mu}+1)}\rho^{2} = \delta_\nu^{2} \,.
\end{equation*}
Thus, $\norm{\Ko x_\nu}\leq \delta_\nu$ and   $ 
  \epsilon(U_{\mu,\rho},\delta_\nu)\geq \norm{x_\nu}\geq \sqrt{{B_{\uo}}/{A_{\vo}}}
  \, \delta_\nu^{2\mu/(2\mu+1)}\rho^{1/(2\mu+1)}  $. \end{proof}

Note that if $\kao$ does not accumulate at zero, then $\Ko^\plus$ is bounded (see Theorem \ref{thm:char}). In this case the inverse problem is well posed and 
\begin{align*}
\begin{aligned}
    \epsilon(U_{\mu,\rho},\delta) &= \sup\{\|x\|\mid x\in U_{\mu,\rho} \wedge \|\Ko x \| \leq \delta\} \\
    &= \sup \{\| \Ko^\plus y\| \mid y \in \Ko(U_{\mu,\rho}) \wedge \|y\| \leq \delta \} \\
    &= \| \Ko^\plus \|.
\end{aligned}
\end{align*}
This reflects that in the well-posed case, where $\Ko^\plus$ is bounded the optimal convergence rate is $\mathcal{O}(\delta)$ independent of particular prior information. In the ill-posed case according to Theorem \ref{thm:opt} this rate is not achievable.

\section{Application to X-ray tomography}
\label{sec:radon}

In this section we  apply  the concept of  filtered DFDs  to X-ray tomography as a prime  example of an inverse problem in medical image reconstruction. In two spatial dimensions, X-ray tomography can be  modeled by the 2D Radon transform.
In this section we study filtered DFDs   for the Radon transform  on $L^2(\R^2)$.   
Throughout this section,  the Fourier transform  of a function $f \in L^1(\R^n)$ is defined by $ \fourier f(\xi) = \int_{\R^n} f(x) e^{-i \inner{\xi}{x}} \dx$  and extended to functions in $L^2 (\R^n)$ by continuity. Its inverse transform is denoted by $\fourier^{-1} $. For functions $g$ defined on $\sph^1 \times \R$  we write $\fourier_2 g$ for the Fourier transform in the second argument.    

\subsection{The Radon transform  on $L^2(\R^2)$}

Wavelet frames are naturally defined on $L^2(\R^2)$. Therefore we will study the Radon transform as an operator on  $L^2(\R^2)$, where it is known to  be unbounded, closed, and densely defined. See~\cite{smith1977practical} for further background. In this  subsection we collect main ingredients for  constructing DFDs and filtered DFDs for the Radon transform. 

\paragraph{Radon transform:} 
Let   $L^2_0(\R^2) \coloneqq \set{f \in L^2(\R^2) \mid  \supp (f) \text{ compact} }$ denote the space of all square integrable functions on $\R^2$ that vanish outside a bounded domain. The 2D Radon transform $\Ro f$ of $f \in   L^2_0(\R^2)$ is defined by
\begin{equation} \label{eq:Rf}
	\forall (\theta,s) \in \sph^1 \times \R \colon \quad \Ro f (\theta,s) = \int_\R f(s \theta + t \theta^\perp) \dt \,.
\end{equation}
The  value $\Ro f (\theta,s)$ is the integral of $f$ over the affine line with  normal  vector $\theta \in \sph^1$ and signed distance $s \in \R$. Given  $f \in   L^2_0(\R^2)$ these integrals are well defined  for almost all  $(\theta,s) $ and yield  an element in  $L^2(\sph^1 \times \R)$ . 

The  Radon transform  can and is extended  to a  densely defined closed  operator 
$\Ro \colon \dom(\Ro) \subseteq L^2(\R^2) \to L^2(\sph^1 \times \R) $ with domain $\dom(\Ro) \coloneqq \{ f \in L^2(\R^2) \mid \norm{\edot}^{-1/2} \fourier f \in L^2(\R^2) \}$. Note that the form  \eqref{eq:Rf} of $\Ro f (\theta,s) $ as line integral does not hold for all $f \in  \dom(\Ro)$. 

\paragraph{Adjoint Radon transform:} The adjoint  $\Ro^* \colon \dom(\Ro^*) \subseteq L^2(\sph^1 \times \R) \to L^2(\R^2)$ 
of the Radon transform has  domain  $\dom(\Ro^*) \coloneqq    \{g \in L^2(\sph^1 \times \R) \mid \abs{\sigma}^{-1/2} \fourier_2 g  \in L^2( \sph^1 \times \R ) \}$, where $\sigma$ is the second argument of $\mathcal{F}_2 g$. One verifies that $\dom(\Ro^*)$ consists of all $g \in L^2(\sph^1 \times \R)$ such that $
	\Ro^\sharp g (x)  \coloneqq  \int_{\mathbb{S}} g(\theta,\inner{x}{\theta}) \dtheta 
$ gives a square integrable function in which case  $\Ro^* g=  \Ro^\sharp g$.   Operator $\Ro^\sharp$ is known as backprojection operator.

\paragraph{Fourier slice theorem:}
The Fourier slice theorem for  $f \in \dom(\Ro)$ reads     
\begin{equation} \label{eq:slice1}
	\fourier_2 \Ro f( \theta, \sigma) = \fourier f(\sigma \theta)  \quad \text{ for a.e.  $( \theta, \sigma) \in  \sph^1 \times \R$ }  \,.
\end{equation}
The Fourier slice identity is  commonly stated  for functions $f \in L^1(\R^2) \supseteq  L^2_0(\R^2)$ in which case $\fourier f$  is a continuous function and   \eqref{eq:slice1} holds point-wise as an easy consequence of the definition of the Radon and Fourier transforms.  Let us verify that \eqref{eq:slice1} indeed also holds on $\dom(\Ro)$. For that purpose note  that  $f \in  \dom (\Ro)$ iff  $\abs{\fourier f}^2$ and $\norm{\edot}^{-1}  \abs{\fourier f}^2$ are integrable. The latter property together with a change of variable  and Fubinis theorem shows  $\int_{\R^2}  \abs{\fourier f(\xi)}^2  \, \norm{\xi}^{-1} \dxi = \int_\mathbb{S} \int_\R \abs{ \fourier f(\sigma \theta)}^2 \dsigma \dtheta $.
Hence the right hand side in \eqref{eq:slice1} is well defined  as an element  of  $L^2(\sph^1 \times \R)$. The same holds true for the left hand side $\fourier_2 \Ro f$. In order that \eqref{eq:slice1} holds true on  $ \dom (\Ro)$ one has to assure that $\fourier_2^{-1} [(\theta,\sigma) \mapsto \fourier f(\sigma \theta) ]$ defines a closed operator on $ \dom (\Ro)$ which is verified  in straight forward manner.

\paragraph{Normal operator:}
The normal  operator  $\Ro^* \Ro$ for the Radon transform is again  densely defined and closed  with  domain $\dom(\Ro^* \Ro)  =  \{f \in L^2(\R^2) \mid \norm{\edot}^{-1} \fourier f \in L^2(\R^2) \}$. The Fourier slice identity \eqref{eq:slice1} and Fubinis theorem  yield the  isometry property    
\begin{equation}\label{eq:slice}
\forall f,g \in \dom(\Ro) \colon  \quad \int_{\sph^1} \int_\R \Ro f(\theta,s) \overline{\Ro g(\theta,s)} \ds \dtheta = 2 \int_{\R^2} \frac{\fourier f(\xi)}{\norm{\xi}}  \; \overline{\fourier g(\xi)} \, \dxi \,.
\end{equation}
The left hand side in \eqref{eq:slice} is the $L^2$-inner product $\inner{\Ro f}{\Ro g}$ which  is equal to $\inner{\Ro^*\Ro f}{g}$ provided that  $ \Ro f \in \dom (\Ro^*)$, or equivalently that $f \in \dom (\Ro^*\Ro)$.  Therefore \eqref{eq:slice} gives the Fourier representation  $  \Ro^* \Ro f  =  2 \fourier^{-1}  ( \norm{\edot}^{-1} \fourier f )$.

\subsection{DFDs for the Radon transform} \label{sec:DFD_Radon}

We now study  DFDs  $(\uo, \vo, \kao)$  for the Radon transform on $L^2(\R^2)$. We first derive necessary properties for   $\vo$ and $\kao$  in the general case and subsequently derive the DFD for the  case that $\uo$ is a wavelet ONB.

\paragraph{Necessary conditions:}

Let $(\uo, \vo, \kao)$ be a DFD for $\Ro$ and assume $v_\la \in \ran(\Ro)$.  Then $v_\la  = \ka_\la \Ro \sigma_\la$ for some $\sigma_\la \in \dom (\Ro)$. By  \ref{D3}, $\Ro^* v_\la =   \kappa_\la  u_\la$ which shows that $u_\la = \Ro^*\Ro \sigma_\la$ and  $\sigma_\la \in \dom (\Ro^*\Ro)$. Equation \eqref{eq:slice} implies $\sigma_\la  = \fourier^{-1}  (  \norm{\edot}  \fourier u_\la ) / 2$  and  therefore 
\begin{equation} \label{eq:vradon}
	v_\la 
	=  \frac{\ka_\la}{2}  \, \Ro \fourier^{-1}  (\norm{\edot} \fourier u_\la)
	=  \frac{\ka_\la}{2}  \,  \Ro \riesz u_\la  \,,    
\end{equation}
where $\riesz  u \coloneqq \fourier^{-1}  (\norm{\edot}  \fourier u)$. 

Next assume that the frame $\uo$ has a multiscale structure  
\begin{equation} \label{eq:wavelets}
	\forall (j,k,\ell) \in \La = \Z \times  \Z^2 \times  L \colon 
	\quad u_{j,k,\ell}(x)  =  2^j u_{0,0,\ell}(2^j x -k) \,.
\end{equation}
Using  \eqref{eq:vradon},   \eqref{eq:slice} and the scaling and translation property of  $\fourier$ show
\begin{align*}
\norm{v_{j,k,\ell}}^2 &= 
\frac{\kappa_{j,k,\ell}^2}{(2\pi)^2} \int_{\sph^1}\int_{\R} \abs{\Ro \riesz  (u_{j,k,\ell})}(\theta, \sigma)^2  \dsigma \dtheta \\
&= 
\frac{\kappa_{j,k,\ell}^2}{(2\pi)^2}  
\int_{\R^2} \frac{ \abs{\norm{\xi}\fourier u_{j,k,\ell} (\xi)}^2 }{\norm{\xi}}  \dxi
\\ &= 
\frac{2^j \kappa_{j,k,\ell}^2 }{(2\pi)^2}  
\int_{\R^2} \frac{ \abs{\norm{\xi}\fourier u_{0,0,\ell} (\xi)}^2 }{\norm{\xi}}  \dxi
\\&= \frac{2^j  \kappa_{j,k,\ell}^2 }{\kappa_{0,0,\ell}^2}   \norm{v_{0,0,\ell}}^2.
\end{align*}
Assuming the frame elements to be bounded away from zero (as it is, for example,  in the case of Riesz bases) this  implies  that the quasi-singular values satisfy  $\kappa_{j,k,\ell} \asymp 2^{-j/2}$.  

The  considerations above show how  to construct a DFD starting with a frame $\uo$   of the form \eqref{eq:wavelets}.  That such a construction actually results  in a DFD  in the case of wavelet ONB has been first shown in the seminal work of Donoho \cite{donoho95nonlinear} and is outlined below.

\paragraph{Wavelet vaguelette decomposition:}

Now let    $\uo = (u_\la)_{\la \in \La}$ be a 2D (tensor product) wavelet ONB for $L^2(\R^2)$ of the  form \eqref{eq:wavelets} where  $\la \in \La = \Z \times \Z^2 \times \{1,2,3\}$ consists of a triples $(j,k,\ell) $, where $j \in \Z$ is the scale index, $k = (k_1,k_2) \in \Z^2$ is the shift index and $\ell \in \{1,2,3\}$ indicates the chosen mother wavelet (horizontal, vertical or diagonal).

\begin{theorem}[Wavelet-vaguelette decomposition \cite{donoho95nonlinear}] \label{thm:wvd}
Let  $\uo \in L^2(\R^2)^\La$  with  $\La = \Z \times \Z^2 \times \{1,2,3\}$ be a 2D wavelet ONB  of the form \eqref{eq:wavelets} such that $u_{0,0,\ell}$ has compact support  and $\norm{\edot} \, \fourier u_{0,0,\ell}  \in L^2(\R^2)$ for $\ell=1,2,3$. Define $\vo$ by  \eqref{eq:vradon} and set $\kappa_\la \coloneqq  2^{-j/2}$. Then $(\uo, \vo, \kao)$ is a diagonal frame decomposition of $\Ro$ with a Riesz basis $\vo$ of  $\overline{\ran \Ro}$. 
\end{theorem}

\begin{proof}
Following the construction of the previous paragraph, the quasi-singular value relations  \ref{D3} are satisfied. It remains to  verify that  $ \vo$ forms a frame of  $\overline{\ran \Ro}$.  For the proof we refer to the original work  of Donoho \cite{donoho95nonlinear}.   He used wavelet-like functions, so-called vaguelettes, which were first introduced by Meyer, for his proof. Therefore he called this particular DFD the wavelet-vaguelette decomposition  (WVD).
\end{proof}

Inspired by the WVD related frame decompositions  for the Radon transform  have been derived where $\uo$ is a curvelet  \cite{candes2002recovering} or a shearlet frame \cite{colonna2010radon}.  

\paragraph{Constructing frame coefficients:}

An essential  ingredient in the actual  implementation of the  filtered DFD, is the efficient computation of the frame  coefficients $\inner{g}{v_{j,k,\ell}} $. For that purpose  we make use of the  explicit expression  \eqref{eq:vradon} which implies    
\begin{equation}
	\inner{g}{v_\la }
	= \frac{\kappa_\la}{2}\inner{g}{ \Ro \riesz u_\la} 
	= \frac{\kappa_\la}{2}\inner{\riesz \Ro^\sharp g}{u_\la}  \,.
\end{equation}
Here $\riesz \Ro^\sharp$ is the filtered backprojection (FBP) inversion formula for the Radon transform.   Since the wavelet transform as well as $\riesz \Ro^\sharp$  can be computed efficiently, this gives also an efficient algorithm for evaluating the coefficients $\inner{g}{v_\la }$.

\begin{figure}[htb!]
\centering
\begin{tikzpicture}[x=1cm, y=1cm, font=\small]
	\centering
	\draw (0,0) node {\includegraphics [scale=0.4]{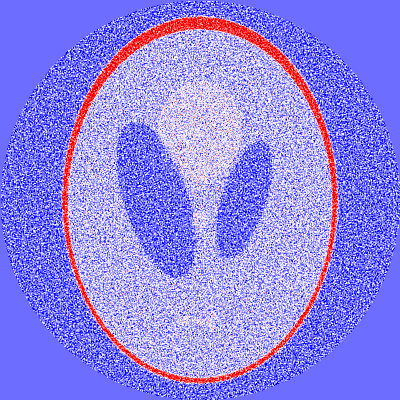}};
	\draw (4.4,0) node {\includegraphics [scale=0.4]{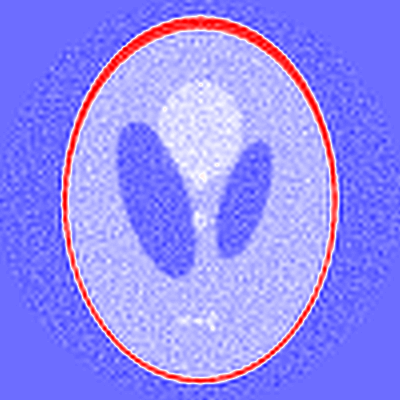}};
	\draw (8.8,0) node {\includegraphics [scale=0.4]{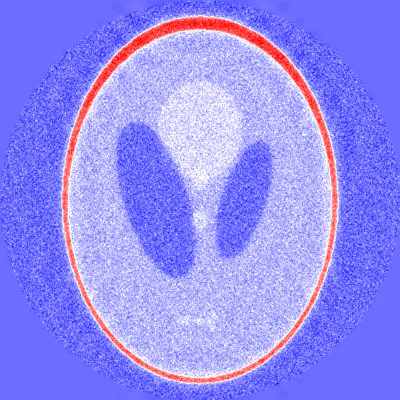}};
	\draw (11.6,0) node {\includegraphics [scale=0.43]{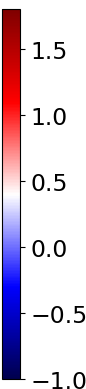}};
\end{tikzpicture}
\caption{Reconstructions using FBP (left), truncated DFD (middle), and  Tikhonv-filtered DFD (right) both with $\al = 0.15^2$.} \label{figure}
\end{figure}

\subsection{Numerical results}

Using the  WVD  $(\uo, \vo, \kao)$ as in Theorem \ref{thm:wvd} as DFD together with the regularizing filters of Subsection~\ref{sec:regfil} we obtain the following two filtered DFD reconstructions  
\begin{align} \label{eq:twvd} 
\Freg_\al^{(1)} g =& \sum_{2^{-j} \geq \al} \inner{\riesz\Ro^\sharp   g}{u_{j,k,\ell}} \,u_{j,k,\ell}  \\ \label{eq:tikwvd}
\Freg_\al^{(2)} g  =& 
 \sum_{j,k,\ell}\frac{2^{-j/2}}{2^{-j}+\al}  \inner{ \riesz \Ro^\sharp  g}{u_{j,k,\ell}}\, u_{j,k,\ell} \, .
\end{align}
We refer to \eqref{eq:twvd}  as  truncated WVD  and to \eqref{eq:tikwvd} as Tikhonov-filtered WVD. All ingredients for evaluating  \eqref{eq:twvd}, \eqref{eq:tikwvd} can be implemented in a straight forward and efficient manner: The  FBP  inversion formula $\riesz \Ro^\sharp g$, the forward and inverse wavelet transform and the coefficient filtering.       

Figure~\ref{figure} shows  reconstructions of the Shepp Logan phantom applied to  Radon transform data with added Gaussian white noise using the FBP reconstruction, truncated WVD and Tikhonov-filtered WVD, respectively.  Table~\ref{table} displays the $\ell^2$-error, the peak-signal-to-noise ratio (PSNR), and the structural similarity index measure (SSIM) of all  reconstructions  for various regularization parameters.

\setlength{\arrayrulewidth}{0.5mm}
\renewcommand{\arraystretch}{1.9}
\setlength{\tabcolsep}{12pt}

\begin{table}[htb!]
\centering
\begin{tabular}{|c|c||c|c|c|}
\hline
Reconstruction method & Parameter & $\ell^2$-error & PSNR & SSIM \\
\hline \hline
FBP & & 0.110 & 63.698 & 0.314 \\
\hline
\multirow{2}{*}{ WVD truncated} & $\al=0.08^2$ & 0.109 & 63.765 & 0.315 \\ \cline{2-5} 
& $\al=0.15^2$ & 0.104 & \textcolor{red}{71.263} & 0.709\\ \cline{2-5}
& $\al=0.25^2$ & 0.223 & 68.426 & \textcolor{red}{0.765} \\ \cline{2-5}
\hline
\multirow{3}{*}{WVD Tikhonov} & $\al=0.08^2$ & \textcolor{red}{0.086} & 67.473 & 0.408 \\ \cline{2-5} 
& $\al=0.15^2$ & 0.125 & 69.75 & 0.573\\ \cline{2-5}
& $\al=0.25^2$ & 0.196 & 68.844 & 0.706 \\ \cline{2-5}  
\hline
\end{tabular}
\caption{Evaluation of reconstruction results using common quality measures. The best results are marked in red (lowest-$\ell^2$ error and highest PSNR and SSIM.)
}\label{table}
\end{table}

\section{Conclusion and outlook}
\label{sec:end}

In this work we analyzed the concept of diagonal frame decomposition (DFD) for the solution of linear inverse problems allowing potentially unbounded forward operators $\Ko$. A DFD for the operator $\Ko$ yields the explicit formula  $\Ko^\plus = (\TT_{\buo}^* \M_{1/\ka} \TT_\vo)|_{\dom(\Ko^\plus)}$ for the Moore-Penrose inverse.  In the ill-posed case, the Moore-Penrose generalized inverse  $\Ko^\plus$ is unbounded as well as is the sequence $1/\ka$. We showed that replacing the $1/\kappa_\la$ by a regularized filter (Definition \ref{def:refF}) applied to the quasi-singular values  $\kappa_\la$ results in a regularization method (Theorem \ref{thm:conv}). As another main result we derived  convergence rates for filtered DFD in Theorem \ref{thm:rates}. By noting that the DFD reduced to the SVD in the case of orthogonal basis, we see that our results extend  convergence and convergence rates results of filter based SVD regularization \cite{engl1996regularization,groetsch1984theory} to the DFD case. We applied  our theory to the inversion of the Radon transform  by filtered DFD as practical application. The Radon transform is unbounded as an operator on  $L^2(\R^2)$ highlighting benefits of including such operators in our theoretical analysis.

One advantage of filtered DFD regularization over variational regularization methods is their explicit form.  Compared to SVD based regularization, benefits are that  a DFD may be available even when no SVD is known or has to be computed numerically.  Moreover,  the associated analysis and synthesis operations can often be implemented efficiently for the DFD.  The use of the DFD is of practical relevance  as frames such as wavelets or curvelets have better approximation capabilities  for typical images to be reconstructed \cite{candes2002recovering} than singular systems. In order to fully exploit such properties a main aspect of future research is to extend the  presented convergence analysis to non-linear filters. As a first step in this direction see the work \cite{frikel2019sparse} where a convergence analysis  is presented using soft-thresholding defining a non-linear filtered DFD.

\section*{Acknowledgment}

The work of AE and MH has been supported by the Austrian Science Fund (FWF), project P 30747-N32. We thank the referees and the editor for the valuable comments on the initial submission that helped to significantly improve our manuscript.

\bibliography{dfd.bib}

\end{document}